\documentclass[a4paper]{amsart}

\usepackage{hyperref}
\hypersetup{
    colorlinks=true,       
    linkcolor=blue,          
    citecolor=blue,        
    filecolor=blue,      
    urlcolor=blue           
}
\usepackage[all]{xy}
\usepackage{tikz,graphicx}
\usepackage{amssymb}
\usepackage{booktabs}
\usepackage{amsmath}

\usepackage{tkz-euclide}

\usetikzlibrary{matrix,arrows,decorations.pathmorphing}

\theoremstyle{plain}
\newtheorem{thm}{Theorem}[section]
\newtheorem{lem}[thm]{Lemma}

\newtheorem{prop}[thm]{Proposition}

\theoremstyle{definition}
\newtheorem{defn}[thm]{Definition}
\usetikzlibrary{matrix,arrows,decorations.pathmorphing}

\def\Aut{\operatorname{Aut}}
\def\PGL{\operatorname{PGL}}

\usepackage{enumerate}

\begin{document}

\title{On quadratic progression sequences on smooth plane curves}
\author[E. Badr] {Eslam Badr}
\address{$\bullet$\,\,Eslam Badr}
\address{Department of Mathematics,
Faculty of Science, Cairo University, Giza-Egypt}
\email{eslam@sci.cu.edu.eg}

\author[M. Sadek] {Mohammad Sadek}
\address{$\bullet$\,\,Mohammad Sadek}
\address{Faculty of Engineering and Natural Sciences,
 Sabanc{\i} University,
  Tuzla, \.{I}stanbul, 34956 Turkey}
\email{mmsadek@sabanciuniv.edu}

\maketitle

\begin{abstract}
We study the arithmetic (geometric) progressions in the $x$-coordinates of quadratic points on smooth projective planar curves defined over a number field $k$. Unless the curve is hyperelliptic, we prove that these progressions must be finite. We, moreover, show that the arithmetic gonality of the curve determines the infinitude of these progressions in the set of $\overline{k}$-points with field of definition of degree at most $n$, $n\ge 3$.
\end{abstract}

\section{Introduction}
Let $C$ be a smooth planar curve defined by a homogeneous polynomial equation $F(X,Y,Z)=0$ over a number field $k$ with an algebraic closure $\overline{k}$. An arithmetic (geometric) progression on $C$ is a set of $\overline{k}$-points on $C$ whose $x$-coordinates are in arithmetic (geometric) progression over $k$. If these points are in $C(k)$, the progression is said to be rational; if they lie in the set of quadratic points $\Gamma_2(C,k)$, it is said to be quadratic; if the field of definition of each point is of degree $\le n$, $n\ge3$, the progression is called an $n$-level progression.

Rational progressions on planar algebraic curves of genus $g\ge 1$ form the natural analogues of arithmetic and geometric progression sequences on the rational line. Consequently, a tempting question that rises up is whether these progressions on curves can be of infinite length. It turns out that these rational progressions must be finite. This follows as a direct consequence of Faltings' finiteness theorem of rational points on curves of genus at least two. Therefore, research has been directed toward the question of how long such a progression can be on an algebraic planar curve.

The problem of finding consecutive rational solutions to a diophantine equation can be traced back to Mohanty \cite{Mohanty}. He studied integral arithmetic progressions on the elliptic curve $y^2 = x^3 + k$. It was Bremner \cite{Bremner} who initiated the study of $\mathbb Q$-rational points in arithmetic progression on elliptic curves. In order to construct elliptic and hyperelliptic curves over $\mathbb Q$ with long rational arithmetic progressions, several theoretical and computational tools have been employed. The interested reader may refer to \cite{CampbellUlas,MacLeod,Ulas}. In the same vein rational geometric progressions on elliptic and hyperelliptic curves over $\mathbb Q$ have been studied. Again, the goal was to construct curves over $\mathbb Q$ with long rational geometric progressions \cite{sadekalaa,BremnerUlas}.

When one considers the set of quadratic points of planar curve, it turns out that these sets might be infinite for certain families of curves. Thus, one may study progressions in these sets in the hope of finding infinite progressions. In this note we tackle the aforementioned problem. For elliptic and hyperelliptic curves, one always finds quadratic progressions of infinite length. However, for other smooth planar curves, we show that despite the infinitude of the set of quadratic points in certain cases, quadratic arithmetic and geometric progressions still resemble rational progressions on smooth planar curves. More precisely, we prove that any quadratic progression on a smooth planar curve, which is neither elliptic nor hyperelliptic, must be finite.

The results of this note can be approached as follows. There are two separate cases that one may examine. If $\deg F \ge 5$, then our finiteness results follow from earlier work of Abramovich, Harris and Silverman. When $\deg F=4$, then the problem needs more analysis as one has to show that the arithmetic of smooth planar quartic curves does not allow the existence of infinite quadratic progressions.

Motivated by the existence of infinite quadratic progressions on elliptic and hyperelliptic curves, we establish a connection between the arithmetic gonality of smooth planar algebraic curves and infinite $n$-level progressions on these curves. More specifically, as long as $n$ is at least of the same size as the arithmetic gonality of the curve and the curve possesses a rational point, it follows that any $n$-level arithmetic or geometric progression on the curve is guaranteed to be of infinite length. For example, a smooth planar quartic curve $C$ over $k$ can never possess an infinite quadratic progression, yet any $n$-level progression on $C$, $n\ge 3$, must be infinite due to the fact that $C$ is always trigonal.

\subsection*{Acknowledgments} The second author is partially supported by Sabanc{\i} University Starting Grant Fund.
\section{Basic Definitions}
By $k$ we always mean a number field with a fixed algebraic closure $\overline{k}$. A smooth projective plane curve over $k$ is a smooth curve $C/k$ that is $k$-isomorphic to the zero locus of a homogenous polynomial equation $F(X,Y,Z)=0$ with coefficients in $k$ and no singularities over $\overline{k}$. We write $C(k)$ to denote the set of $k$-rational points of $C$. We write $\Gamma_2(C,k)$ for the set of quadratic points of $C/k$, i.e.,
\[\Gamma_2(C,k)=\bigcup_{\substack{\ell\\ [\ell:k]=2}}C(\ell).\]

\begin{defn}[Progression sequences]\label{defn1}
Let $C: F(X,Y,Z)=0$ be a projective plane curve over $k$. A sequence $P_i:=(x_i,y_i,z_i)$, for $i=1,2,\ldots$, of points in $C(k)$ (resp. $\Gamma_2(C,k)$), is called a \emph{rational (resp. quadratic) geometric progression sequence on $C$} if $x_i,y_i\in k$ such that $\{x_i\,|\,i=1,2,\ldots\}$ form a geometric progression in the base field $k$, that is to say there exist $t,t'\in k^*$ with $x_i=t't^i$, for $i=1,2,\ldots$.

Similarly, we may define \emph{rational (resp. quadratic) arithmetic progression sequence on $C$} if $\{x_i\,|\,i=1,2,\ldots\}$ form an arithmetic progression in the base field $k$, that is to say there exist $t,t'\in k^*$ with $x_i=t'+it$, for $i=1,2,\ldots$.
\end{defn}
We can always work affine by setting $Y=1$ because the $Y$-coordinate of a point in a progression sequence is, by definition, rational. We also note that a rational progression sequence on $C$ is by definition a quadratic progression sequence on $C$. The converse does not need to be true as will be seen later, see Lemma \ref{lem:hyperelliptic}.
\begin{defn}
Let $P_i:=(x_i,y_i,z_i)$, $i=1,2,\ldots,m$, be a finite geometric (resp. arithmetic) progression sequence on $C$. The integer $m$ is called the \emph{length} of the sequence.
\end{defn}

\section{Rational progressions on smooth plane curves}
Although experts may be familiar with the results in this section, we prefer to exhibit these results with their proofs in order for the paper to be self-contained.

\begin{prop}
\label{prop:rational}
Let $C:F(X,Y,Z)=0$, $\deg (F)\ge 3$, be a smooth projective plane curve over a number field $k$ with genus $g\ge 1$. Any rational geometric (resp. arithmetic) progression sequence on $C$ must be finite.
\end{prop}
\begin{proof}
If the genus of $C$ is $g\ge 2$, then by Faltings' celebrated result, \cite{Fa1}, the set of rational points $C(k)$ is finite. In particular, any subset of $C(k)$ is finite, hence the result.

We are left with the case when the genus $g$ of the curve $C$ is $1$. We consider the following two subcases. (i) There is an infinite rational arithmetic progression $(x_i,y_i,z_i)$ in $C(k)$, where $x_i=t'+i t$, $i=1,2,\cdots$, and $t,t'\in k$: One considers the curve $C'$ defined by $$F'(x,y,z)=F(t'+tx,y,z)=0.$$ This curve has infinitely many rational points $(i,y_i,z_i)$, $i=1,2,\cdots$. We consider the subsequence of rational points $(k^{5},y_{k^5},z_{k^5})$, $k=1,2,\cdots$. Given that there is a $k$-morphism of degree $5$, namely, $C'\to C$, $(x,y,z)\mapsto(x^5,y,z)$, and that there are at least $10$ ramification points for the morphism, Riemann-Hurwitz formula implies that the curve $C'$ is of genus $g\ge 2$ with infinitely many rational points $(k,y_{k^5},z_{k^5})$ contradicting the fact that it must have finitely many rational points.
(ii) There is an infinite rational geometric progression $(x_i,y_i,z_i)$ in $C(k)$, where $x_i=t' t^i$, $i=1,2,\cdots$, $t,t'\in k$: The proof is similar, see \cite[Theorem 2]{sadekalaa}.
\end{proof}

One may pose the question of the existence of a uniform bound that depends on the degree of $F$ and $[k:\mathbb Q]$ for the number of elements in a rational geometric (resp. arithmetic) progression sequence on $C:F(x,y,z)=0$. Another problem would be providing explicit examples of polynomials $F(x,y,z)$ of a fixed degree $d$ such that $C$ possesses a rational progression of a given length. Furthermore, the finiteness result above motivates us to question the size of quadratic progression sequences on smooth plane curves which is the subject of this note.

\section{Quadratic progressions on smooth plane curves}

In contrast to rational progressions on smooth projective curves, quadratic progressions can be infinite. We need first to recall the following definition.

\begin{defn} A smooth projective curve $C$ is called \emph{hyperelliptic} (resp. \emph{bielliptic}) \emph{over $k$} if there exists a degree two $k$-morphism from $C$ to the projective line $\mathbb{P}^1_k$ (resp. to an elliptic curve $E$) over $k$.
\end{defn}

\begin{lem}\label{lem:hyperelliptic}
Given an algebraic curve $C$ defined over $k$ by $z^2=f(x,y)$ where $f$ is homogeneous with $\deg(f)\ge 2$, there exists an infinite quadratic geometric (resp. arithmetic) progression sequence on $C$.
\end{lem}
\begin{proof}
We have a canonical $2-1$ morphism to a projective line $\mathbb{P}^1_k$ over $k$ given by $(x,1,z)\mapsto (x,1)$. Thus, it suffices to consider any  geometric (resp. arithmetic) progression $x_i$, for $i=1,2,\ldots$ over $k$ of infinite length, and it will follow that $(x_i,1,\sqrt{f(x_i)})$ is indeed an infinite length progression on $C$.
\end{proof}
Lemma \ref{lem:hyperelliptic} gives an example of a quadratic geometric (resp. arithmetic) progression sequence, which is not rational, since on hyperelliptic curves rational geometric (resp. arithmetic) progression sequences are of finite length due to Proposition \ref{prop:rational}.

In view of a key result of Abramovich-Harris in \cite{abha} and Harris-Silverman in \cite{HaSi}, one knows exactly which curves one has to visit to look for infinite quadratic progression sequences.
\begin{thm}\label{AbramovichHarrisSilvermanI}
Let $C$ be a smooth projective curve over $k$ of geometric genus $g\geq 2$. The set $\Gamma_2(C,k)$ is an infinite set if and only if $C$ is
hyperelliptic over $k$, or bielliptic over $k$ such that there exists a degree two $k$-morphism from $C$ to an elliptic curve $E$ of positive rank over $k$.
\end{thm}
As a direct consequence, one obtains the following result.
\begin{thm}\label{thm1}
Given a smooth projective plane curve $C: F(X,Y,Z)=0$ of degree $d\geq5$ over $k$, a quadratic geometric (resp. arithmetic) progression sequence on $C$ is always finite.
\end{thm}
\begin{proof}
One easily can show that a smooth projective plane curve of degree $d\geq5$ is neither hyperelliptic nor bielliptic over $k$ (cf. \cite[\S 2]{BaBaquadraticbielliptic}). Consequently, $\Gamma_2(C,k)$ is a finite set by the aid of Theorem \ref{AbramovichHarrisSilvermanI}, in particular, any quadratic geometric (resp. arithmetic) progression sequence on $C$ should be finite.
\end{proof}
Theorem \ref{thm1} may be used to recover Proposition \ref{prop:rational} when $\deg(F)\ge 5$. Namely, a rational geometric (resp. arithmetic) progression sequence on a smooth plane curve $C: F(X,Y,Z)=0$ of degree $d\geq5$ over $k$ is finite.

According to Lemma \ref{lem:hyperelliptic} and Theorem \ref{thm1}, one sees that the curves of interest for us right now are those described by $F(X,Y,Z)=0$, where $\deg(F)=4$.


The set $\Gamma_2(C,k)$ of quadratic points of a smooth plane quartic curve $C$ may be infinite (cf. \cite[Theorems 3.2 and 3.6]{BaBaquadraticbielliptic}). At first sight, one may expect to find a quadratic geometric (resp. arithmetic) progression sequence of infinite length on such curves. In this section, we will show that $\Gamma_2(C,k)$ contains no such quadratic progression sequences when $C$ is a smooth plane quartic curve.

\begin{thm}\label{conjecture1}
 A quadratic geometric (resp. arithmetic) progression sequence on a smooth plane quartic curve $C$ over a number field $k$ is always of finite length.
\end{thm}
\begin{proof}
If $\Gamma_2(C,k)$ is finite, then the statement is trivial. So we assume that $\Gamma_2(C,k)$ is infinite. It follows that either $C$ is hyperelliptic over $k$ or bielliptic over $k$, see Theorem \ref{AbramovichHarrisSilvermanI}. Since a smooth plane quartic curve is never hyperelliptic, \cite[Exercise IV.3.2]{Harts}, the curve $C$ is bielliptic over $k$. In particular, there exists an involution $\omega\in\Aut(C\times\overline{k})$ having $4$ fixed points on $C$. There is no loss of generality
to assume that $\omega$ is defined by $X\mapsto X$, $Y\mapsto Y$, $Z\mapsto -Z$ (this is true up to $\PGL_3(k')$-projective equivalence for some finite field extension $k'/k$. So, one may, and will, replace the field $k$ with the extension $k'$ by noticing that $C\times_kk'$ is bielliptic over $k'$ and $\Gamma_2(C\times k',k')$ still infinite as it contains $\Gamma_2(C,k)$). Now, $C$ may be described by an equation of the form \[aZ^4+L_2(X,Y)Z^2+L_{4}(X,Y)=0,\] where $a\in k^*$ and $L_{i}(X,Y)$ is a homogenous polynomial of degree $i$ in $k[X,Y]$. This is equivalent to saying that there is a $k$-morphism of degree $2$ from $C$ to the elliptic curve $E:aZ'^2+L_2(X',1)Z'=-L_4(X',1)$ defined by $(X,1,Z)\mapsto (X',1,Z')=(X,1,Z^2)$ and $E$ is of positive rank over $k$.

  Now, we assume on the contrary that $C$ has an infinite quadratic geometric progression sequence $P_i=(x_i,1,z_i)\in\Gamma_2(C,k)$, where $x_i=t't^i$, $i=1,2,3,\cdots$, $t',t\in k $ whereas $z_i\in k_i$ and $[k_i:k]=2$. Therefore, one obtains
\[az_i^4+L_2(t't^i,1)z_i^2+L_4(t't^i,1)=0.\]
For a fixed integer $s>1$, we consider the subsequence $(t't^{si},1,z_{si})$, $i=1,2,\cdots$. One has the smooth curve $C_s$ with the following plane (possibly singular) model \[C_s:az^4+L_2(t'x^s,1)z^2+L_4(t'x^s,1)=0\] over $k$ with the infinite quadratic geometric progression $(t^i,1,z_{si})$. In particular, this gives rise to an infinite tower of smooth curves
\[\cdots\xrightarrow{\phi_{n}} C_{2^n}\xrightarrow{\phi_{n-1}} C_{2^{n-1}}\xrightarrow{\phi_{n-2}}\cdots \xrightarrow{\phi_2}C_4\xrightarrow{\phi_1}C_2\xrightarrow{\phi_0}C_1=C\]
where $\phi_n:C_{2^{n+1}}\to C_{2^{n}}$ is a $k$-morphism of degree $2$ defined by $(x,1,z)\mapsto(x',1,z')=(x^2,1,z)$. Moreover, each of the curves $C_i$ admits an infinite quadratic geometric progression, namely, $(t^i,1,z_{2^ni})\in \Gamma_2(C_{2^n},k)$. Therefore, according to Theorem \ref{AbramovichHarrisSilvermanI}, each $C_{2^i}$, $i=1,2,\cdots,$ is either hyperelliptic or bielliptic. 
In what follows we rule out these two possibilities, hence such infinite quadratic progression can never exist.
\begin{itemize}
\item[(i)] Assume that $C_{2^i}$ is hyperelliptic, for some $i$, $i\ge1 $. If we consider the nonconstant morphism $C_{2^i}\to C$, then our assumption implies that $C_{2^i}$ admits a $k$-morphism of degree $2$ to the projective line $\mathbb{P}_k^1$, hence $C$ admits a $k$-morphism of degree $\le 2$ to $\mathbb{P}_k^1$, see statement $S(2,0)$ in \cite[p. 227]{abha}. In particular, $C$ has gonality $\leq2$, which contradicts the fact that $C$ is a smooth planar quartic curve. Hence $C_{2^i}$ is never hyperelliptic.
    \item[(ii)] Assume that $C_{2^i}$ is bielliptic for every $i\ge 1$. We consider the following diagram

   \begin{center}
   \begin{tikzpicture}[scale=1.5]
   \node(A0) at (0,1) {$\cdots$};
\node (A1) at (1,1) {$C_{2^n}$};
\node (A2) at (2,1) {$C_{2^{n-1}}$};
\node (A3) at (3,1) {$\cdots$};
\node (A4) at (4,1) {$C_{4}$};
\node (A5) at (5,1) {$C_2$};
\node (A6) at (6,1) {$C$};
\node (B0) at (0,0) {$\cdots$};
\node (B1) at (1,0) {$H_{2^n}$};
\node (B2) at (2,0) {$H_{2^{n-1}}$};
\node (B3) at (3,0) {$\cdots$};
\node (B4) at (4,0) {$H_{4}$};
\node (B5) at (5,0) {$H_2$};
\node (B6) at (6,0) {$E$};
\path[->,font=\scriptsize,>=angle 90]
(A0) edge node[above]{$\phi_{n}$} (A1)
(A1) edge node[above]{$\phi_{n-1}$} (A2)
(A2) edge node[above]{$\phi_{n-2}$} (A3)
(A3) edge node[above]{$\phi_{2}$} (A4)
(A4) edge node[above]{$\phi_{1}$} (A5)
(A5) edge node[above]{$\phi_{0}$} (A6)
(A1) edge node[right]{$j_n$} (B1)
(A2) edge node[right]{$j_{n-1}$} (B2)
(A4) edge node[right]{$j_2$} (B4)
(A5) edge node[right]{$j_1$} (B5)
(A6) edge node[right]{$j_0$} (B6)
(B0) edge node[above]{$\psi_{n}$} (B1)
(B1) edge node[above]{$\psi_{n-1}$} (B2)
(B2) edge node[above]{$\psi_{n-2}$} (B3)
(B3) edge node[above]{$\psi_{2}$} (B4)
(B4) edge node[above]{$\psi_{1}$} (B5)
(B5) edge node[above]{$\psi_{0}$} (B6);
\end{tikzpicture}
    \end{center}
The morphism $j_0:C\to E$ is the $k$-morphism of degree $2$ describing the bielliptic nature of $C$. The $k$-morphism $j_i:C_{2^i}\to H_{2^i}$, $i=1,2,\cdots,$ is given by $(x,1,z)\mapsto (x',1,z')=(x,1,z^2)$ where the hyperelliptic curve $H_{2^i}$ is described by $az^2+L_2(t'x^{2^i},1)z=-L_4(t'x^{2^i},1)$; whereas the morphism $\psi_i:H_{2^{i+1}}\to H_{2^{i}}$ is described by $(x,1,z)\mapsto(x',1,z')=(x^2,1,z)$.

The morphism $C_{2^i}\to H_{2^i}$ with $C_{2^i}$ being bielliptic yields that $H_{2^i}$ is bielliptic itself, this is statement $S(2,1)$, \cite[p. 227]{abha}. By Riemann-Hurwitz formula, one knows that the genus of $H_{2^i}$ is strictly larger than the genus of $H_{2^{i-1}}$. Thus one reaches a contradiction by noticing that Castelnuovo-Severi inequality prevents a hyperelliptic curve of genus $g> 3$ from being bielliptic, see for instance \cite{accola}.
\end{itemize}
The case of quadratic arithmetic progression sequences can be treated in a similar fashion.
\end{proof}

\begin{section}{Higher level progressions on smooth plane curves}
We recall that the {\em arithmetic gonality}, $\operatorname{gon}(C,k)$, of an algebraic curve $C$ over $k$ is defined to be the lowest degree of a morphism from $C$ to the projective line. We may generalize Definition \ref{defn1} as follows.
\begin{defn}
Let $C$ be a projective plane curve defined by $F(X,Y,Z)=0$ over $k$. A sequence $P_i:=(x_i,y_i,z_i)$, for $i=1,2,\ldots$, of points in $\Gamma_n(C,k)$, is called an {$n$-level progression sequence on $C$} if $x_i,y_i\in k$ such that $\{x_i\,|\,i=1,2,\ldots\}$ form a progression in the base field $k$. Here $\Gamma_n(C,k)$ denotes the set of $n$-points of $C$, that is the set of all points on $C$ with field of definition of degree at most $n$ over $k$.
\end{defn}

\begin{thm}
Let $C$ be a smooth plane curve defined by $F(X,Y,Z)=0$ over $k$, where $F$ is of degree $d\geq4$, such that $C(k)\neq\emptyset$. Then $C$ admits an $n$-level progression of infinite length for any  $n\geq\operatorname{gon}(C,k)$.
\end{thm}
\begin{proof}
Given a smooth plane curve $C$ over a perfect field $k$, the geometric gonality for $C$ equals $d-1$ and every geometric gonal map is a projection from a point on $C\times_k\overline{k}$ (cf. \cite{Na, Yoshihara}). Consequently, if we further assume that $C(k)\neq\emptyset$, then the arithmetic gonality for $C$ over $k$ is equal to its geometric gonality and so any $k$-gonal map $f:C\rightarrow \mathbb{P}^1_k$ is a projection from $C$ to the projective line $\mathbb{P}^1_k$ centered at a $k$-point of $C$ (in particular, $k$-gonal maps are finitely many because $C(k)$ is a finite set due to Faltings' \cite{Fa1}).

We may assume that the reference point $(0:0:1)$ belongs to $C: F(X,Y,Z)=0$, 
in particular, $C$ is defined, up to $\PGL_3(k)$-equivalence, by \[Z^{d-1}X+\text{lower order terms in}\,Z\,\text{over}\,k.\] Let $(x_i,1)$, for $i=1,2,\ldots$, be any infinite rational progression of $\mathbb{P}^1_k$. Choose, for each $i=1,2,\ldots$, a point $(x_i,1,z_i)\in f^{-1}(x_i,1)$. Consequently, the minimal irreducible polynomial $\operatorname{Irr}_{z_i,k}(z)$ of $z_i$ over $k$ divides $F(x_i,1,z)$ in $k[z]$, hence $(x_i,1,z_i)$ is defined over some field extension $k_i$ of degree $\operatorname{deg}(\operatorname{Irr}_{z_i,k}(z))\leq d-1$ over $k$. This yields an infinite $(d-1)$-level progression for $C$, which was to be shown.

\end{proof}
\end{section}

\end{document}